\documentclass[a4paper,oneside,11pt]{amsart}

\reversemarginpar

\usepackage[colorlinks,hyperindex,linkcolor=blue,urlcolor=black,pdftitle={On the quantity of  shift-type invariant subspaces which span the whole space},pdfauthor={Maria F.Gamal'}]
{hyperref}

\usepackage{amsmath}
\usepackage{amsfonts}
\usepackage{amssymb}
\usepackage{amsthm}

\usepackage[english]{babel}
\usepackage{enumerate}

\usepackage{graphicx}
\usepackage{color}

\usepackage[abbrev]{amsrefs}

\numberwithin{equation}{section}

\theoremstyle{plain}
\newtheorem{theorem}{Theorem}[section]
\newtheorem{lemma}[theorem]{Lemma}

\newtheorem{corollary}[theorem]{Corollary}
\newtheorem{proposition}[theorem]{Proposition}

\theoremstyle{definition}
\newtheorem{remark}[theorem]{Remark}
\newtheorem{example}[theorem]{Example}

\begin{document}

\title[Shift-type invariant subspaces]{On polynomially bounded operators with  shift-type invariant subspaces}

\author{Maria F. Gamal'}
\address{
 St. Petersburg Branch\\ V. A. Steklov Institute 
of Mathematics\\
 Russian Academy of Sciences\\ Fontanka 27, St. Petersburg\\ 
191023, Russia  
}
\email{gamal@pdmi.ras.ru}





\begin{abstract} A particular case of  \cite{ker07} was generalized from contractions to polynomially bounded operators 
in \cite{gam19}. Namely, it is proved in  \cite{gam19} that if the unitary asymptote of a polynomially bounded operator $T$ 
 contains the bilateral shift of multiplicity $1$,  then there exists an invariant subspace $\mathcal M$ of $T$ 
such that $T|_{\mathcal M}$ is similar to the unilateral shift of multiplicity $1$. In the present paper, some corollaries of this result 
are given. In particular, reflexivity of polynomially bounded operators described above is proved.\end{abstract}

\maketitle

\section{Introduction}

Let $\mathcal H$ be a (complex, separable) Hilbert space, and let $\mathcal M$ be its
(linear, closed)  subspace.
By $I_{\mathcal H}$ and $P_{\mathcal M}$ the identical operator on $\mathcal H$ and 
the orthogonal projection from $\mathcal H$ onto $\mathcal M$ are denoted, respectively.  

Let $\mathcal L(\mathcal H)$ be the algebra of all (linear, bounded) operators  acting on $\mathcal H$. 
For an operator  $T\in \mathcal L(\mathcal H)$, a subspace $\mathcal M$  of $\mathcal H$ is called \emph{invariant subspace of $T$}, 
if $T\mathcal M\subset\mathcal M$. The complete lattice of all invariant  subspaces of $T$ is denoted by  $\operatorname{Lat}T$. 
The smallest algebra containing $T$ and $I_{\mathcal H}$ which is closed in the weak operator topology is denoted by $\operatorname{Alg}T$. 
The algebra of all $A\in \mathcal L(\mathcal H)$ such that $\operatorname{Lat}T\subset\operatorname{Lat}A$ 
is denoted by $\operatorname{Alg}\operatorname{Lat}T$. An operator  $T\in \mathcal L(\mathcal H)$ is called \emph{reflexive}, if 
$\operatorname{Alg}\operatorname{Lat}T=\operatorname{Alg}T$. Every subnormal operator is reflexive 
(see {\cite[Sec.VII.8]{con}} and references therein), in particular, 
every isometry is reflexive \cite{deddens}. 

For $T\in\mathcal L(\mathcal H)$ and $\{x_k\}_{k=1}^n\subset\mathcal H$ set 
$$\mathcal E_T(\{x_k\}_{k=1}^n)=\vee_{k=1}^n\vee_{j=0}^\infty T^j x_k.$$
The \emph{multiplicity} $\mu_T$ of an operator  $T\in \mathcal L(\mathcal H)$ 
 is the minimum dimension of its reproducing subspaces: 
$$  \mu_T=\min\{n\ : \ \mathcal E_T(\{x_k\}_{k=1}^n) = \mathcal H, \ \ \{x_k\}_{k=1}^n\subset\mathcal H \},$$
where $1\leq n\leq\infty$.
An operator $T$ is called {\it cyclic}, if $\mu_T=1$. 

It is well known and easy to see that if $\mathcal M\in\operatorname{Lat}T$, then 
\begin{equation}\label{muorth} \mu_{P_{\mathcal M^\perp}T|_{\mathcal M^\perp}}\leq\mu_T\leq
  \mu_{T|_{\mathcal M}} + \mu_{P_{\mathcal M^\perp}T|_{\mathcal M^\perp}}  \end{equation}
and \begin{equation}\label{muker}\mu_T\geq \dim\ker T^\ast  \end{equation}
(see, for example, {\cite[II.D.2.3.1]{nik}}).

Let $\mathcal H$ and $\mathcal K$ be Hilbert spaces, and let  $\mathcal L(\mathcal H,\mathcal K)$ be the space of  (linear, bounded) transformations 
  acting from $\mathcal H$ to $\mathcal K$.  
Let $T\in \mathcal L(\mathcal H)$, let $R\in \mathcal L(\mathcal K)$, and let
$X\in \mathcal L(\mathcal H,\mathcal K)$ be a transformation such that 
$X$ {\it intertwines} $T$ and $R$,
that is, $XT=RX$. If $X$ is unitary, then $T$ and $R$ 
are called {\it unitarily equivalent}, in notation:
$T\cong R$. If $X$ is invertible, that is, $X$ has the {\it bounded} 
inverse $X^{-1}$, 
then $T$ and $R$ 
are called {\it similar}, in notation: $T\approx R$. 
It is well known and easy to see that similarity preserves many properties of operators, in particular, reflexivity. 
If $X$ is a {\it quasiaffinity}, that is, $\ker X=\{0\}$ and 
$\operatorname{clos}X\mathcal H=\mathcal K$, then
$T$ is called a {\it quasiaffine transform} of $R$, 
in notation: $T\prec R$. If $T\prec R$ and 
$R\prec T$,
 then $T$ and $R$ are called {\it quasisimilar}, 
in notation: $T\sim R$. 
If $\operatorname{clos}X\mathcal H=\mathcal K$, we write $T \buildrel d \over \prec R$. 
It  follows immediately from the definition of the relation $\buildrel d \over \prec$ that if $T \buildrel d \over \prec R$,
then $\mu_R\leq \mu_T$.

 $\mathbb D$ and $\mathbb T$ denote the open unit disc
and the unit circle, respectively. The normalized Lebesgue measure on $\mathbb T$ is denoted by $m$.   
 $H^\infty$ denotes the Banach algebra of all bounded 
analytic functions in $\mathbb D$ with the uniform norm $\|\cdot\|_\infty$ on $\mathbb D$.  
 $H^2$ denotes the Hardy space on $\mathbb D$, and $L^2=L^2(\mathbb T,m)$.
 The simple unilateral and bilateral shifts $S$ and $U_{\mathbb T}$ are the operators
 of multiplication by the independent variable  on $H^2$ and on $L^2$, respectively. 
It is well known that $\mu_S=1$ and $\mu_{U_{\mathbb T}}=1$. 

An operator $T$ is called a \emph{contraction}, if $\|T\|\leq 1$. An operator $T$ is called {\it polynomially bounded}, if there exists a constant $C$ such that $$\|p(T)\|\leq C\|p\|_\infty \text{ \ for every polynomial } p.$$ 
The smallest such $C$ is called \emph{the polynomial bound of $T$} and is denoted by $C_{{\rm pol},T}$ there. 
 Every contraction $T$  is polynomially bounded with $C_{{\rm pol},T}=1$ by the von Neumann inequality (see, for example, {\cite[Proposition I.8.3]{sfbk}}).

If $T$ is a polynomially bounded operator, then $T=T_a\dotplus T_s$,
where $T_a$ is an  \emph{absolutely continuous (a.c.)}  polynomially bounded operator, 
that is, the $H^\infty$-functional calculus is well defined for $T_a$, and 
$T_s$ is similar to a singular unitary operator, see \cite{mlak} or \cite{ker16}. (Although many results on polynomially 
bounded operators that will be used in the present paper were originally  proved by Mlak, we will refer to \cite{ker16} for the convenience of references.)  
 Clearly,  $S$ and $U_{\mathbb T}$ are isometries. Consequently, they are contraction, therefore, they are  polynomially bounded. 
It is well known that $S$ and $U_{\mathbb T}$ are a.c.. Recall that 
\begin{equation}\label{ssrefl}\operatorname{Alg}\operatorname{Lat}S=\operatorname{Alg}S=
\{\varphi(S)\ :\ \varphi\in H^\infty\},\end{equation}
see, for example, {\cite[Theorem I.4.11]{con}} or {\cite[Theorem IX.3.8]{sfbk}}. 

The following lemma  follows from the definition of a.c. polynomially bounded operators (see {\cite[Lemma 2.1]{gam17}} for a detailed proof).

\begin{lemma}\label{lem1new} Suppose that $T\in\mathcal L(\mathcal H)$ is a polynomially bounded operator and 
$\mathcal H=\mathcal H_a\dotplus \mathcal H_s$, where $\mathcal H_a$, $\mathcal H_s\in\operatorname{Lat}T$, 
 $T|_{\mathcal H_a}$ is a.c. and  $T|_{\mathcal H_s}$  is similar to a singular unitary operator. Let  
$\mathcal M\in\operatorname{Lat}T$. If $T|_{\mathcal M}$ is a.c., then $\mathcal M\subset \mathcal H_a$.\end{lemma}

The following lemma is a straightforward corollary of Lemma \ref{lem1new}. 

\begin{lemma}\label{lem2new} Suppose that $T\in\mathcal L(\mathcal H)$  is a polynomially bounded operator, 
$\{\mathcal M_n\}_n\subset\operatorname{Lat}T$,  $\mathcal H=\bigvee_n\mathcal M_n$, 
and $T|_{\mathcal M_n}$ is a.c. for every $n$. Then $T$ is a.c.. \end{lemma}

An  operator $T\in\mathcal L(\mathcal H)$ is called  \emph{power bounded}, if $ \sup_{j\geq 0}\|T^j\| < \infty$. Classes $C_{ab}$, where $a$, $b=0, 1, \cdot$, of power bounded operators 
are defined as follows (see {\cite[Sec. II.4]{sfbk}} and \cite{ker89}). 
 $T$ is \emph{of class} $C_{0\cdot}$, if $\|T^jx\|\to 0$ when $j\to\infty$ for every $x\in\mathcal H$. 
$T$ is \emph{of class} $C_{1\cdot}$, if $\inf_{j\geq 0}\|T^jx\|> 0$ for every $0\neq x\in\mathcal H$. $T$ is 
\emph{of class} $C_{\cdot a}$, if $T^*$ is of class $C_{a\cdot}$, and $T$ is \emph{of class} $C_{ab}$,  
if $T$ is of class $C_{a\cdot}$ and  of class  $C_{\cdot b}$.

For a power bounded operator $T$ the \emph{isometric asymptote} $(X_+,T_+^{(a)})$  can be  defined using a Banach limit, 
see \cite{ker89} and \cite{ker16}. Here $T_+^{(a)}$ is an isometry (on a Hilbert space), and 
$X_+$ is the \emph{canonical intertwining mapping}: $X_+T=T_+^{(a)}X_+$.  
The \emph{unitary asymptote} $(X,T^{(a)})$  is a pair where $ T^{(a)}$ is the minimal unitary extension of $T_+^{(a)}$  and $X$ is an extension of $X_+$.  

 Let $\nu$ be a  finite positive Borel measure on $\operatorname{clos}\mathbb D$. Denote by $P^2(\nu)$ and by $S_\nu$ 
the closure of (analytic) polynomials in $L^2(\operatorname{clos}\mathbb D,\nu)$ and the operator  of multiplication by 
the independent variable  on $P^2(\nu)$, respectively. Operators $S_\nu$ are subnormal, see \cite{con}. 
Clearly, $S_m=S$.
\begin{equation}\label{congss}\text{ If  }  w\in L^1(\mathbb T,m), \ w>0 \text{ a.e. on }  \mathbb T \text{  and } \log w\in L^1(\mathbb T,m),  \text{  then } S_{wm}\cong S,\end{equation} 
see {\cite[Ch. III.12, VII.10]{con}} or {\cite[I.A.4.1]{nik}}. 

The  following lemma is a corollary of \cite{bour}, a detailed proof can be found in  \cite{bercpr} or \cite{gam19}. 

\begin{lemma}[\cite{bour},\cite{bercpr}]\label{lembt} Suppose that $T\in\mathcal L(\mathcal H)$ is an a.c. polynomially bounded operator, and 
$x\in\mathcal H$. Then there exists a function $w\in L^1(\mathbb T,m)$ such that 
$w\geq 0$, 
and \begin{equation}
\label{polw21}
\|p(T)x\|^2\leq \int_{\mathbb T}|p|^2w\text{\rm d}m\ \ \text{ for every polynomial } p.\end{equation}\end{lemma}

\begin{corollary}\label{corss}Suppose that $T\in\mathcal L(\mathcal H)$ is an a.c. polynomially bounded operator, and 
$x\in\mathcal H$. Then $S\buildrel d\over\prec T|_{\mathcal E_T(x)}$.
\end{corollary}

\begin{proof}Let $w$ be from Lemma \ref{lembt} applied to $x$.  Take $\varepsilon>0$ and set 
$w_\varepsilon=\max(w,\varepsilon)$. Clearly \eqref{polw21} is fulfilled for $w_\varepsilon$ instead of $w$. 
Furthermore,  \eqref{polw21} with  $w_\varepsilon$ means that 
$S_{ w_\varepsilon m} \buildrel d\over\prec  T|_{\mathcal E_T(x)}$.
 By \eqref{congss},  $S_{ w_\varepsilon m}\cong S$. \end{proof}

\section{Reflexivity}

 Let $T$ be  a polynomially bounded operator (on a Hilbert space),  and let $T\buildrel d\over\prec U_{\mathbb T}$. 
By {\cite[Theorem 4]{ker89}}, 
$\mathbb T\subset \sigma(T)$. By \cite{rej}, $T$ either has a nontrivial hyperinvariant subspace or is reflexive. 
The reflexivity of  contractions $T$ such that  $T\buildrel d\over\prec U_{\mathbb T}$ is proved in  
\cite{tak87} and \cite{tak89}, see also {\cite[Theorem IX.3.8]{sfbk}}. The proof   
can be generalized to polynomially bounded operators. 

For a polynomially bounded operator $T$ set
\begin{equation}\label{bound1}C_{{\rm sim} S,T}=\bigr(\sqrt 2(K^2+2)KC_{{\rm pol},T}+1\bigl)\sqrt{K^2C_{{\rm pol},T}^2+1}KC_{{\rm pol},T}^2,\end{equation}
where $C_{{\rm pol},T}$ is the polynomial bound of $T$ and $K$ is a universal constant from \cite{bour}.

\begin{lemma}\label{lemnew1} Suppose that  $T\in\mathcal L(\mathcal H)$ is an a.c. polynomially bounded operator, $x\in\mathcal H$, 
$\mathcal M_0\in\operatorname{Lat}T$, $Y_0\in\mathcal L(H^2,\mathcal M_0)$ is invertible, $Y_0S=T|_{\mathcal M_0}Y_0$. Then 
for every $c>0$ there exists $\mathcal N\in\operatorname{Lat}T$ and $Y\in\mathcal L(H^2,\mathcal N)$ such that 
$YS=T|_{\mathcal N}Y$, $\|Y\|\leq(1+c)\|Y_0\|$, $\|Y^{-1}\|\leq(1+c)\|Y_0^{-1}\|$ and $\mathcal M_0\vee \mathcal N=\mathcal M_0\vee \mathcal E_T(x)$.
\end{lemma} 

\begin{proof}  By Corollary \ref{corss},  there exists $W\in\mathcal L(H^2,\mathcal E_T(x))$ such that $WS=T|_{\mathcal E_T(x)}W$ and 
$\operatorname{clos}W H^2=\mathcal E_T(x)$.Take $\varepsilon\neq 0$ and set 
$$ Y f = Y_0 f + \varepsilon W f, \ \ \ f\in H^2, \ \ \ \text{ and } \ \ \ \mathcal N = YH^2.$$
Clearly, $Y$ and $\mathcal N$ satisfy to the conclusion of the lemma for sufficiently small $\varepsilon$. 
\end{proof}

\begin{theorem}\label{thmnew1} Suppose that  $T\in\mathcal L(\mathcal H)$ is an a.c. polynomially bounded operator, 
$\mu_T$ is the multiplicity of $T$, and $T\buildrel d \over\prec U_{\mathbb T}$.
Then for every $c>0$ there exist $\{\mathcal M_k\}_{k=0}^{\mu_T}\subset\operatorname{Lat}T$ 
and $Y_k\in\mathcal L(H^2,\mathcal M_k)$ such that $Y_kS=T|_{\mathcal M_k}Y_k$, 
$\|Y_k\|\|Y_k^{-1}\|\leq (1+c)C_{{\rm sim} S,T}$ for all $k$, 
and $\bigvee_{k=0}^{\mu_T}\mathcal M_k=\mathcal H$. Here $1\leq\mu_T\leq\infty$.
\end{theorem}

\begin{proof} Take $0<c_1<c$. By \cite{gam19}, there exist 
$\mathcal M_0\in\operatorname{Lat}T$ and  $Y_0\in\mathcal L(H^2,\mathcal M_0)$ such that $Y_0$ is invertible, $Y_0S=T|_{\mathcal M_0}Y_0$, 
and $\|Y_0\|\|Y_0^{-1}\|\leq (1+c_1)C_{{\rm sim} S,T}$. Let $\{x_k\}_{k=1}^{\mu_T}\subset \mathcal H$ be such that 
$\mathcal E_T(\{x_k\}_{k=1}^{\mu_T})= \mathcal H$. Let $Y_k$ and $\mathcal M_k$ be constructed in Lemma \ref{lemnew1} applied for every $x_k$ 
with $k\geq 1$. It is easy to see that $\{\mathcal M_k\}_{k=0}^{\mu_T}$ and $\{Y_k\}_{k=0}^{\mu_T}$ satisfy to the conclusion of the theorem.
\end{proof}

\begin{remark}Let in assumptions of Theorem \ref{thmnew1} $T$ be a contraction. Applying 
 \cite{ker07} or {\cite[Theorem IX.3.6]{sfbk}} instead of  \cite{gam19}, one can obtain the conclusion of Theorem  \ref{thmnew1} with $1$ instead of $C_{{\rm sim} S,T}$.    
\end{remark}


The following theorem is a corollary of Theorem  \ref{thmnew1} and is proved exactly as {\cite[Theorem 1]{tak89}}
or  {\cite[Theorem IX.3.8]{sfbk}}. We give a detailed proof for reader's convenience. 

\begin{theorem} \label{thmrefla} Suppose that  $T$ is an a.c.  polynomially bounded operator, and $T\buildrel d \over\prec U_{\mathbb T}$. 
Then $\operatorname{Alg\, Lat}T = \{\varphi(T)\ : \ \varphi\in H^\infty\}$. Consequently, $T$ is reflexive.
\end{theorem}

\begin{proof} Denote by $\mathcal H$ the space in which $T$ acts. By Theorem \ref{thmnew1}, there exist $\frak N\subset\mathbb N$, $\{\mathcal M_k\}_{k\in\frak N}\subset\operatorname{Lat}T$, and 
 $\{Y_k\}_{k\in\frak N}$ 
such that $Y_k$ are invertible and $Y_kS =T|_{\mathcal M_k}Y_k$ for every $k\in\frak N$, and 
\begin{equation}\label{spanmmk}\bigvee_{k\in\frak N}\mathcal M_k=\mathcal H. \end{equation}
Let $A\in\operatorname{Alg\, Lat}T$. Then $ \mathcal M_k\in\operatorname{Lat}A$ and 
$A|_{\mathcal M_k}\in\operatorname{Alg\, Lat}T|_{\mathcal M_k}$ for every $k\in\frak N$. Since $T|_{\mathcal M_k}\approx S$, 
there exist $\{\varphi_k\}_{k\in\frak N}\subset H^\infty$ such that $A|_{\mathcal M_k} = \varphi_k(T|_{\mathcal M_k})$ for every $k\in\frak N$ by \eqref{ssrefl}. 

Let $k$, $n\in\frak N$. If $\mathcal M_k\cap\mathcal M_n\neq\{0\}$, then $\varphi_k(T|_{\mathcal M_k\cap\mathcal M_n})=\varphi_n(T|_{\mathcal M_k\cap\mathcal M_n})$. 
Since $T|_{\mathcal M_k\cap\mathcal M_n}\approx S$, we conclude that $\varphi_k=\varphi_n$.

Suppose that $\mathcal M_k\cap\mathcal M_n=\{0\}$. 
There exists $0\neq c\in\mathbb C$ such that $Y\in\mathcal L(H^2,\mathcal H)$ defined by the  formula
$Yf=Y_kf+cY_nf$, $f\in H^2$, is left-invertible. Set $\mathcal M=YH^2$, then $\mathcal M\in\operatorname{Lat}T$ and $T|_{\mathcal M}\approx S$. 
Therefore, there exists $\varphi\in H^\infty$ such that $A|_{\mathcal M} = \varphi(T|_{\mathcal M})$. Let $0\neq f\in H^2$.  We have
\begin{align*}AYf & =\varphi(T)Yf  =\varphi(T)(Y_kf+cY_nf)=Y_k\varphi f+cY_n\varphi f \\
\text{ and } AYf & = A(Y_kf+cY_nf) = \varphi_k(T)Y_kf + \varphi_n(T)cY_nf = Y_k\varphi_k f + cY_n\varphi_n f.  \end{align*}
Therefore, $Y_k(\varphi-\varphi_k)f=cY_n(\varphi_n-\varphi) f$. Since $\mathcal M_k\cap\mathcal M_n=\{0\}$, 
we conclude that $\varphi=\varphi_k$ and $\varphi_n=\varphi$. 

We obtain that $\varphi_k=\varphi_n$ for every $k$, $n\in\frak N$. Therefore, there exists $\varphi\in H^\infty$ such that
$A|_{\mathcal M_k} = \varphi(T|_{\mathcal M_k})$ for every $k\in\frak N$. By \eqref{spanmmk}, $A=\varphi(T)$.
\end{proof}

The following theorem is a corollary of Theorem \ref{thmrefla}.

\begin{theorem} Suppose that  $T$ is a polynomially bounded operator, and $T\buildrel d \over\prec U_{\mathbb T}$. Then $T$ is reflexive.
\end{theorem}

\begin{proof} Denote by $\mathcal H$ the space in which $T$ acts. We have $\mathcal H=\mathcal H_a\dotplus\mathcal  H_s$ and 
$T=T_a\dotplus T_s$, where $T_a=T|_{\mathcal H_a}$ is a.c. and $T_s=T|_{\mathcal H_s}$ is similar to 
a singular unitary operator $V$, see \cite{mlak} or \cite{ker16}. 
 It is well known that 
\begin{equation}\label{plus} \operatorname{Alg}T=\operatorname{Alg}T_a\dotplus \operatorname{Alg}T_s. \end{equation}
We sketch the proof briefly. Denote by $\mathcal P$ the (skew) projection onto $\mathcal H_s$ along $\mathcal H_a$.
Denote by $\nu$ the scalar-valued spectral measure of $V$. There exists 
 a  family $\{\sigma_k\}_{k\in\mathbb N}$ of compact subsets of $\mathbb T$ such that $m(\sigma_k)=0$, 
$\sigma_k\subset\sigma_{k+1}$ for every $k\in\mathbb N$ and $\nu(\cup_{k\in\mathbb N}\sigma_k)=\nu(\mathbb T)$.
For every $k\in\mathbb N$ there exists $\varphi_k$ from the disk algebra such that $\varphi_k=1$ on $\sigma_k$ 
and $|\varphi_k|<1$ on $\operatorname{clos}\mathbb D\setminus\sigma_k$ 
(see, for example, {\cite[Ch.6, p.81]{hoffman}}). We have 
$ \lim_k \lim_n \varphi_k^n(T)=\mathcal P$. Therefore, $\mathcal P \in\operatorname{Alg}T$. Equality \eqref{plus} 
follows from the latter inclusion. By \eqref{plus} and \cite{conwaywu}, $T$ is reflexive if and only if $T_a$ and $T_s$ are reflexive. 
Since $T_s$ is similar to a unitary operator, $T_s$ is reflexive.

Since there is no nonzero transformation intertwining a.c. and singular unitaries, 
we conclude that $T_a\buildrel d \over\prec  U_{\mathbb T}$.
 By Theorem \ref{thmrefla}, $T_a$ is reflexive. \end{proof}

\section{On the quantity of  shift-type invariant subspaces which span the whole space}

In this section, it is shown that  if $T\in\mathcal L(\mathcal H)$ is an a.c. polynomially operator with $2\leq\mu_T<\infty$, then the quantity of   shift-type invariant subspaces which span $\mathcal H$ is $\mu_T$. But the estimates of the norms of intertwining transformations become worse than in Theorem \ref{thmnew1}.  

The following simple lemmas are given for convenience of references. Therefore, their proofs are omitted.

\begin{lemma}\label{lemmult} Suppose $n\in\mathbb N$, $n\geq 2$,  $T\in\mathcal L(\mathcal H)$ is an a.c. polynomially bounded operator, 
$x_1, \ldots, x_n\in\mathcal H$, and $\varphi_2,\ldots,\varphi_n\in H^\infty$. 
Then $$\mathcal E_T(x_1, \ldots, x_n)= \mathcal E_T\Big(x_1+\sum_{k=2}^n \varphi_k(T) x_k, \ x_2, \ldots, x_n\Big).$$
\end{lemma}

\begin{lemma}\label{lemsupp2} Suppose $f$, $g\in L^2$, and $\varphi\in H^\infty$ is such that 
$$|\varphi(\zeta)|=\begin{cases} 1, & \text{ if } \ |f(\zeta)| \neq |g(\zeta)|,\\ 2, & \text{ if } \ |f(\zeta)| = |g(\zeta)|.
\end{cases}$$
Then $\operatorname{ess}  \operatorname{supp} f\cup \operatorname{ess}  \operatorname{supp} g=\operatorname{ess}  \operatorname{supp}(f+\varphi g)$.\end{lemma}

The following lemma can be easy proved by induction with applying of Lemma \ref{lemsupp2}. Therefore, its proof is omitted.

\begin{lemma}\label{lemsuppn} Suppose  $n\in\mathbb N$, $n\geq 2$,  and $f_1,\ldots, f_n\in L^2$. 
Then there exist $\varphi_2,\ldots,\varphi_n\in H^\infty$ 
such that $$\cup_{k=1}^n\operatorname{ess}  \operatorname{supp} f_k 
=\operatorname{ess}  \operatorname{supp} ( f_1+\sum_{k=2}^n \varphi_k f_k).$$ \end{lemma}

The following theorem is a corollary of  Lemma \ref{lem2new} and the definition of multiplicity.  Therefore, its proof is omitted.

\begin{theorem} Suppose that  $n\in\mathbb N$, $T\in\mathcal L(\mathcal H)$ is a polynomially bounded operator,  
there exist $\mathcal M_1, \ldots, \mathcal M_n\in\operatorname{Lat}T$ such that 
$T|_{\mathcal M_k}\approx S$, $k=1, \ldots, n$, and $\bigvee_{k=1}^n\mathcal M_k=\mathcal H$. 
Then $\mu_T\leq n$ and $T$ is a.c..\end{theorem}

 \begin{lemma}\label{lemcyclic} Suppose that $n\in\mathbb N$, $n\geq2$, 
  $T\in\mathcal L(\mathcal H)$, $\mathcal M$, $\mathcal N\in\operatorname{Lat}T$, 
$\mathcal M\subset\mathcal N$, $\{h_k\}_{k=1}^n\subset\mathcal M$,  $\{u_k\}_{k=1}^n\subset\mathcal N\ominus\mathcal M$, $ \{v_k\}_{k=2}^n\subset\mathcal N^\perp$, 
\begin{equation}\label{cyclicnn}\mathcal E_T( h_1\oplus u_1\oplus 0 )=\mathcal N \end{equation} and  
\begin{equation}\label{cyclic}\mathcal E_T(h_1\oplus u_1\oplus 0, \ h_k\oplus u_k\oplus v_k, \ \ k=2, \ldots, n)= \mathcal H.\end{equation}
 Then for every  $h_0\in \mathcal M$ 
$$\mathcal E_T(h_0\oplus u_1\oplus v_2, ( h_0-h_1)\oplus 0\oplus v_k, \ k=2, \ldots, n)= \mathcal H.$$\end{lemma}

\begin{proof} Denote by $\mathcal E$ the space from the left side of the conclusion of the lemma. Clearly, 
$$ h_1\oplus u_1\oplus 0=(h_0\oplus u_1\oplus v_2)-((h_0-h_1)\oplus 0\oplus v_2)\in\mathcal E.$$
By the latter inclusion and \eqref{cyclicnn}, $\mathcal N\subset\mathcal E$. 
Therefore, $h\oplus u\oplus v_k\in\mathcal E$ 
for $k=2, \ldots, n$ and for every $h\in \mathcal M$ and $u\in\mathcal N\ominus\mathcal M$. It remains to apply \eqref{cyclic}.\end{proof}

\begin{proposition}\label{propnew1} Suppose that $n\in\mathbb N$, $C>0$,   $T\in\mathcal L(\mathcal H)$ is an a.c. polynomially bounded operator, $U$ is an a.c. unitary operator, $X$ is a transformation such that  $XT=UX$,  
$\mathcal M\in\operatorname{Lat}T$, $\|Xh\|\geq \|h\|/C$ for every $h\in\mathcal M$, 
 $U|_{X\mathcal M}\cong S$, and $\{x_k\}_{k=1}^n\subset \mathcal H$. Then for every $c>0$ there exists $h \in\mathcal M$ such that 
$$\|Xx\|\geq \frac{\|x\|}{(1+c)\sqrt{2(C^2+1)}}$$ for every $ x\in\mathcal E_T(h+x_k)$ and every $1\leq k\leq n$, 
and $U|_{X\mathcal E_T(h+x_k)}\cong S$.
\end{proposition}
\begin{proof} Set $\mathcal K_0=\vee_{j=1}^\infty U^{-j}X\mathcal M$. Without loss of generality we can assume that 
$\mathcal K_0=L^2$, $U|_{\mathcal K_0}= U_{\mathbb T}$, and $X\mathcal M=H^2$. 
Let $V\in\mathcal L(\mathcal K)$ be an a.c.  unitary operator such that $U=U_{\mathbb T}\oplus V$. 
(Note that $V$ can be the  zero operator acting on  the zero space.)

 $T$ and $X$ have the following forms with respect to the decompositions 
$\mathcal H=\mathcal M\oplus \mathcal M^\perp $ and $L^2\oplus\mathcal K$:
$$T=\begin{pmatrix} T|_{\mathcal M} & A \\ \mathbb O & T_0 \end{pmatrix}
\ \ \text{ and } \ \ 
X=\begin{pmatrix} X|_{\mathcal M} & B \\ \mathbb O & X_0 \end{pmatrix}, $$
where $T_0$, $A$, $X_0$, $B$ are appropriate operators and transformations. 

Let $h$, $h_0\in\mathcal M$ and  $y\in\mathcal M^\perp$. Set $f=Xh$, $f_0=Xh_0$ and  $g=By$.  Then 
$f$, $f_0\in H^2$ and $g\in L^2$. Let $w\in L^1(\mathbb T,m)$ be the function from Lemma \ref{lembt} 
applied to $T$ and $y$. We have 
\begin{align*} \|Xp(T)\bigl((h+h_0)\oplus y\bigr)\|^2 &\geq\int_{\mathbb T}|p|^2|f+f_0+g|^2\text{\rm d}m  \\ 
\text{ and }\ \|p(T)\bigl((h+h_0)\oplus y\bigr)\|^2 & \leq 2\Bigl(C^2  \int_{\mathbb T}|p|^2|f+f_0|^2\text{\rm d}m + 
\int_{\mathbb T}|p|^2 w\text{\rm d}m \Bigr)
\end{align*} 
 for every polynomial $p$. 
 
Take $\varepsilon>0$.  If 
\begin{equation}\label{fnew}|f|\geq |f_0|+\max\Bigl(w^{1/2},\varepsilon, \frac {1+c}{c}|g|\Bigr)\ \ \text{ a.e. on  }\mathbb T, \end{equation}
then  \begin{equation}\label{lognew} \log|f+f_0+g|\in L^1(\mathbb T,m)\end{equation}
and $$\frac{1}{(1+c)\sqrt{2(C^2+1)}}\|p(T)\bigl((h+h_0)\oplus y\bigr)\| \leq \|Xp(T)\bigl((h+h_0)\oplus y\bigr)\| $$
 for every polynomial $p$. 
Set $\mathcal E=\mathcal E_T\bigl((h+h_0)\oplus y\bigr)$. We have $X\mathcal E\in\operatorname{Lat}U$, 
 and 
$$\operatorname{clos}P_{L^2}X\mathcal E=\mathcal E_{U_{\mathbb T}}\bigl(P_{L^2}X\bigl((h+h_0)\oplus y\bigr)\bigr)=
\mathcal E_{U_{\mathbb T}}(f+f_0+g).$$
By \eqref{lognew}, $U_{\mathbb T}|_{\operatorname{clos}P_{L^2}X\mathcal E}\cong S$. Thus, 
 $U|_{X\mathcal E}\buildrel d \over\prec S$. Since $U|_{X\mathcal E}$ is a cyclic a.c. isometry, 
we conclude that $U|_{X\mathcal E}\cong S$. 

For $1\leq k\leq n$ set $h_k=P_{\mathcal M}x_k$, $y_k=P_{\mathcal M^\perp}x_k$, $f_k=Xh_k$, $g_k=By_k$, and denote by $w_k$ 
 the function from Lemma \ref{lembt} applied to $T$ and $y_k$. There exists $f\in H^2$ which satisfies  \eqref{fnew} 
with $f_k$, $w_k$, $g_k$ for every $1\leq k\leq n$. There exists $h\in\mathcal M$ such that $Xh=f$.  
It is easy to see that $h$ satisfies  the conclusion of the proposition.  
\end{proof}

\begin{theorem}\label{thmnew2} Suppose that  $T\in\mathcal L(\mathcal H)$ is an a.c. polynomially bounded operator, 
 $T\buildrel d \over\prec U_{\mathbb T}$, and $2\leq \mu_T< \infty$.
Then for every $c>0$ there exist $\{\mathcal M_k\}_{k=1}^{\mu_T}\subset\operatorname{Lat}T$ 
and $Y_k\in\mathcal L(H^2,\mathcal M_k)$ such that $Y_kS=T|_{\mathcal M_k}Y_k$, 
$\|Y_k\|\|Y_k^{-1}\|\leq (1+c)\sqrt{2(C_{{\rm sim} S,T}^2+1)}C_{{\rm pol},T}$  for all $1\leq k \leq\mu_T$,  
and $$\bigvee_{k=1}^{\mu_T}\mathcal M_k=\mathcal H.$$ 
\end{theorem}

\begin{proof} By {\cite[Theorem 2]{ker89}}, there exists a unitary operator $V$ such that 
$T^{(a)}\cong U_{\mathbb T}\oplus V$. Without loss of generality we can assume that $T^{(a)}= U_{\mathbb T}\oplus V$. 
 Denote by  $X$  the canonical intertwining mapping for $T$ and $T^{(a)}$ constructed using a Banach limit.
 Set $n=\mu_T$.  
Let $x_1, \ldots, x_n\in\mathcal H$ be such that $\mathcal E_T(x_1, \ldots, x_n) =\mathcal H$. 
Set $g_{0k}=P_{L^2}X x_k$, $k=1, \ldots, n$. We have 
$\cup_{k=1}^n\operatorname{ess}  \operatorname{supp} g_{0k} =\mathbb T$. 
Let $\varphi_2,\ldots,\varphi_n\in H^\infty$ be functions from Lemma \ref{lemsuppn} applied to 
$g_{01}, \ldots, g_{0n}$. 
Set $$x_0=x_1+\sum_{k=2}^n \varphi_k(T)x_k\ \ \text{ and }\ \ \mathcal N=\bigvee_{j=0}^\infty T^j x_0=\mathcal E_T(x_0).$$ 
Clearly, $P_{L^2}Xx_0 = g_{01}+\sum_{k=2}^n \varphi_k g_{0k}$. 
By Lemma \ref{lemsuppn}, $\operatorname{ess}  \operatorname{supp} P_{L^2}Xx_0=\mathbb T$. 
Therefore, $T|_{\mathcal N}\buildrel d\over\prec U_{\mathbb T}$.

Take $0<c_1<c$.  By \cite{gam19}, there exist $\mathcal M\in\operatorname{Lat}T$  and 
$W\in\mathcal L(H^2,\mathcal M)$     such that $\mathcal M\subset\mathcal N$,  $WS=T|_{\mathcal M}W$, and 
$\|W\|\|W^{-1}\|\leq (1+c_1)C_{{\rm sim} S,T}$. We have
$$\|(T|_{\mathcal M})^jh\|=\|WS^jW^{-1}h\|\geq\frac{\|h\|}{\|W\|\|W^{-1}\|}\geq
\frac{\|h\|}{ (1+c_1)C_{{\rm sim}S,T}}$$
for every $h\in\mathcal M$.  By the construction of $X$ (see \cite{ker89} or \cite{ker16}), 
$\|X\|\leq C_{{\rm pol},T}$ and $\|Xx\|\geq\lim\inf_j\|T^jx\|$ for every $x\in\mathcal H$. Therefore, 
$$\|Xh\|\geq\frac{\|h\|}{ (1+c_1)C_{{\rm sim}S,T}} \text{ for every } h\in \mathcal M.$$ 

By Lemma \ref{lemmult}, $\mathcal E_T(x_0, x_2, \ldots, x_n)=\mathcal H$. 
We have $x_0 = h_1\oplus u_1\oplus 0$ and $x_k=h_k\oplus u_k\oplus v_k$ for $2\leq k\leq n$ with respect to the decomposition 
$$\mathcal H =\mathcal M\oplus (\mathcal N\ominus\mathcal M) \oplus \mathcal N^\perp.$$ 
Let $h_0$ be from Proposition \ref{propnew1} applied to $0\oplus u_1\oplus v_2$ and $(-h_1)\oplus 0\oplus v_k$, $k=2, \ldots, n$, with $C= (1+c_1)C_{{\rm sim}S,T}$ and sufficiently small $c_2>0$. 
Set $$\mathcal M_1 = \mathcal E_T(h_0\oplus u_1\oplus v_2) \  \ \text{ and } \ \  
\mathcal M_k = \mathcal E_T\bigl(( h_0-h_1)\oplus 0\oplus v_k\bigr), \ \ \ k=2, \ldots, n. $$ 
Taking into account Lemma \ref{lemcyclic}, we obtain that  $\{\mathcal M_k\}_{k=1}^n$ satisfy the conclusion of the theorem.
\end{proof}

\begin{remark}Let in assumptions of Theorem \ref{thmnew2} $T$ be a contraction. Applying 
 \cite{ker07} or {\cite[Theorem IX.3.6]{sfbk}} instead of  \cite{gam19}, one can obtain the conclusion of Theorem  \ref{thmnew2} with $2$ instead of $\sqrt{2(C_{{\rm sim} S,T}^2+1)}C_{{\rm pol},T}$.    
\end{remark}

\section{Examples}

It is well known that if $T$ is a contraction and  $T^{(a)}\cong U_{\mathbb T}$, then the multiplicity $\mu_T$ can be arbitrary. Here we recall simplest examples in which the computation of $\mu_T$ is based on \eqref{muorth} and \eqref{muker}.

The following lemma is a generalization of {\cite[Corollary 6]{ker16}}. We sketch the proof. 

\begin{lemma}\label{lemoplus}
Let  $\{T_k\}_{k=1}^N$ be a family of power bounded operators, 
and let $$ \sup_k \sup_{j\geq 0}\|T_k^j\|<\infty.$$
Then $\bigl(\oplus_{k=1}^N T_k\bigr)^{(a)}\cong\oplus_{k=1}^N T_k^{(a)}$. Here $N\in\mathbb N$ or $N=\infty$.
\end{lemma}

\begin{proof} Let $\mathop{\rm Lim}$ be a Banach limit, and let $X_k$ be the canonical intertwining mappings for 
$T_k$ and $ T_k^{(a)}$ constructed using
 $\mathop{\rm Lim}$ for every $k$.  Set $X=\oplus_{k=1}^NX_k$. Denote by $\mathcal H_k$ the spaces in which $T_k$ act.  Let $x\in \oplus_{k=1}^N \mathcal H_k$. Then $x= \oplus_{k=1}^Nx_k$, where 
$x_k \in\mathcal H_k$. We have 
$$\{\|(\oplus_{k=1}^N T_k)^j x\|^2 \}_{j\geq 0}= \sum_{k=1}^N\{\|T_k^jx_k\|^2\}_{j\geq 0} \text{ in }\ell^\infty,$$
and if $N=\infty$, then the convergence in the above equality is in $\|\cdot\|_{\ell^\infty}$. Therefore, 
\begin{align*}\|Xx\|^2&=\sum_{k=1}^N\|X_kx_k\|^2=
\sum_{k=1}^N\mathop{\rm Lim}\bigl(\{\|T_k^jx_k\|^2\}_{j\geq 0}\bigr)
\\&= \mathop{\rm Lim}\bigl(\{\|(\oplus_{k=1}^N T_k)^j x\|^2\}_{j\geq 0}\bigr)
\geq \mathop{\rm lim\, inf}_j\|(\oplus_{k=1}^N T_k)^j x\|^2.
\end{align*}
The lemma follows from  the latter estimate and {\cite[Proposition 1]{ker16}}. 
\end{proof}

\begin{example}\label{exa0} Let $T_0$ and $T_1$ be power bounded operators such that 
$T_0$ is  of class  $C_{0\cdot}$, and   $(T_1)^{(a)}\cong U_{\mathbb T}$.  
Set $T=T_0\oplus T_1$. By Lemma \ref{lemoplus}, $T^{(a)}\cong U_{\mathbb T}$. By  \eqref{muorth}, $\mu_T\geq\mu_{T_0}$. \end{example}

\begin{example}\label{exa1}  

 Let $\nu$ be a finite positive  Borel measure on $\operatorname{clos}\mathbb D$. The space $P^2(\nu)$ and the operator $S_\nu$ are defined in Introduction.   Clearly, $S_\nu$ is a contraction.  
Let $J_\nu\colon P^2(\nu)\to L^2(\nu|_{\mathbb T})$ act by the formula $J_\nu f= f|_{\mathbb T}$, 
$f\in P^2(\nu)$. Denote by $V_{\nu|_{\mathbb T}}$  the operator  of multiplication by 
the independent variable  on $L^2(\nu|_{\mathbb T})$.  
It follows directly from the construction of the unitary asymptote given in \cite{ker89} that 
\begin{equation}\label{nuasymp}(J_\nu, V_{\nu|_{\mathbb T}}) \text{  is the unitary asymptote of }  S_\nu. \end{equation}  

Recall that $m$ is the normalized Lebesgue measure on $\mathbb T$.   
Let $m_2$ be the normalized area measure on $\mathbb D$. Denote by $  \text{\bf 1}$ the unit constant function on 
$\operatorname{clos}\mathbb D$, that is, $  \text{\bf 1}(z)=1$ ($z\in\operatorname{clos}\mathbb D$). 
Let $\mathcal I\subset\mathbb T$ be an open arc. 
Set $$\nu_{\mathcal I}=m_2+m|_{\mathcal I}.$$ Let $f\in P^2(\nu_{\mathcal I})$. Then $f$ is analytic in $\mathbb D$, 
and $f$ has nontangential boundary values a.e. on $\mathcal I$ with respect to $m$, which coincide with 
$f|_{\mathcal I}$ (see, for example, \cite{milsmith}, \cite{milsmithyang} and references therein). Therefore, 
\begin{equation}\label{c10}S_{\nu_{\mathcal I}}\ \text{ is of class } C_{10}.\end{equation} 
 Furthermore, since functions from  $P^2(\nu_{\mathcal I})$ are analytic in $\mathbb D$, $  \text{\bf 1}\not\in S_{\nu_{\mathcal I}}P^2(\nu_{\mathcal I})$ and $S_{\nu_{\mathcal I}}P^2(\nu_{\mathcal I})$ is closed (see, for example, 
\cite{alers} or \cite{gam09}). Thus, \begin{equation}\label{dimker}\dim\ker S_{\nu_{\mathcal I}}^\ast=1.\end{equation}  

 Let $\{\mathcal I_k\}_{k=1}^N$ be a family of open arcs of $\mathbb T$ such that $\mathcal I_k\cap\mathcal I_n=\emptyset$, if $k\neq n$, and 
\begin{equation}\label{fullmeasure} \sum_{k=1}^N m(\mathcal I_k)=1.\end{equation} Here $N\in\mathbb N$ or $N=\infty$.
Set $$T=\bigoplus_{k=1}^N S_{\nu_{\mathcal I_k}}.$$
By \eqref{c10}, $T$ is of class $C_{10}$. By Lemma \ref{lemoplus}, \eqref{nuasymp}, and \eqref{fullmeasure}, 
$T^{(a)}\cong U_{\mathbb T}$. By \eqref{dimker},  \eqref{muorth} and \eqref{muker}, 
$\mu_T=N$.   

\end{example}

\section{Similarity to $S$ and $U_{\mathbb T}$}

In Sec. 2 and  3 we consider shift-type invariant subspaces with estimates of the norms of intertwining transformations 
in terms of the polynomial bound.  
In this section we show that we cannot expect that every invariant subspace satisfying natural necessary conditions will be 
shift-type even any  prescribed estimates of the norms of intertwining transformations are not required. 

Set  $\chi(\zeta)=\zeta$ and $\text{\bf 1}(\zeta)=1$ ($\zeta\in\mathbb T$). 
Let  $\vartheta\in H^\infty$ be an inner function. Set 
$$\mathcal K_\vartheta =H^2\ominus\vartheta H^2 \ \ \ \ \text{and } \ \ \ \ 
T_\vartheta=P_{\mathcal K_\vartheta}S|_{\mathcal K_\vartheta}.$$
Then $\mathcal K_\vartheta=\vartheta \overline\chi\overline{\mathcal K_\vartheta}$, $T_\vartheta$ is an a.c. contraction, and $\vartheta(T_\vartheta)=\mathbb O$. 

\subsection{Quasiaffine transforms of $S$}

 \begin{lemma}\label{lemasymp} Suppose that $T$ is a polynomially bounded operator and $T\prec S$. 
Then $T$ is a.c. and $T_+^{(a)}\cong S$.\end{lemma}

\begin{proof} 
By \cite{bercpr},  there exists a contraction $R$ such that $R\prec T$.  
Clearly, $R\prec S$. By {\cite[Proposition 9]{kers}} or  {\cite[Lemma 2.1]{gam12}}, $R_+^{(a)}\cong S$. 
Denote by  $X_R$ and $X_T$ the canonical intertwining mappings for $R$ and $T$, 
and by $Y$ and $W$ the quasiaffinities which realize the relations $R\prec T$ and $T\prec S$, respectively. 
By {\cite[Theorem 1(a)]{ker89}}, there exist  operators $Z_1$  and $Z_2$ such that $Z_1S=T_+^{(a)}Z_1$,  $X_T Y=Z_1 X_R$, 
$Z_2T_+^{(a)}=SZ_2$, and $W=Z_2X_T$.  Since the ranges of $X_T$ and $W$ are dense, the 
ranges of $Z_1$ and $Z_2$ are also dense. We obtain $S\buildrel d\over\prec T_+^{(a)}$ and $T_+^{(a)}\buildrel d\over\prec S$. 
Since $T\prec S$, $T$ is a.c. by \cite{mlak} or {\cite[Proposition 16]{ker16}}.  Therefore, $T_+^{(a)}$ is a.c..  
Thus, $T_+^{(a)}\cong S$. \end{proof}

 \begin{lemma}\label{leminvers} 
Suppose that $g$, $\varphi$, $\psi$, $\vartheta\in H^\infty$, $\vartheta$ is inner, and 
$1=g\varphi+\vartheta\psi$. Furthermore, suppose that $T\in\mathcal L(\mathcal H)$ is a polynomially bounded operator, $T\prec S$,
$X$ is the canonical intertwining mapping for $T$ and $S$, and 
$Y\in\mathcal L(H^2,\mathcal H)$ is  such that $XY=g(S)$. 
Set $\mathcal M=\operatorname{clos}\vartheta(T)\mathcal H$. If $T|_{\mathcal M}\approx S$, then $T\approx S$.\end{lemma}

\begin{proof} 
By Lemma \ref{lemasymp}, $T$ is a.c.. Thus, the operator $\vartheta(T)$ is well defined. 
Since $X\vartheta(T)=\vartheta(S)X$ and $\operatorname{clos}X\mathcal H = H^2$, we conclude that 
$\operatorname{clos}X\mathcal M = \vartheta H^2$. Suppose that $T|_{\mathcal M}\approx S$. By {\cite[Theorem 1]{ker89}}, 
$X|_{\mathcal M}$ is left invertible, in particular, $X\mathcal M=\operatorname{clos}X\mathcal M = \vartheta H^2$. 
Let $f\in H^2$. There exists $x\in\mathcal M$ such that $\vartheta f=Xx$. We have 
\begin{align*} f&=(g\varphi+\vartheta \psi)f=g(S)(\varphi f) + \psi(S)(\vartheta f)= (XY)(\varphi f) +\psi(S)Xx\\ 
&=X\bigl(Y(\varphi f )+ \psi(T)x\bigr).\end{align*}
We conclude that $X \mathcal H = H^2$. Since $\ker X=\{0\}$, $X$ is invertible by the Closed Graph Theorem. \end{proof}

 \begin{corollary}\label{corinvers} Suppose that $T$ is a polynomially bounded operator and $T\prec S$. 
If  $T|_{\mathcal M}\approx S$ for every $\mathcal M\in\operatorname{Lat}T$ such that $\dim \mathcal M^\perp=\infty$, 
then $T\approx S$.\end{corollary}

\begin{proof} Denote by $\mathcal H$ the space in which $T$ acts, and by  $X$ the canonical intertwining mapping for $T$ and $S$. 
Note that $X$ is a quasiaffinity.
By Corollary \ref{corss}, there exists nonzero $Y\in\mathcal L(H^2,\mathcal H)$ such that $YS=TY$. Since $(XY)S=S(XY)$,  there exists $g\in H^\infty$ 
such that $XY=g(S)$. Since $\ker X=\{0\}$, we have $g\not\equiv 0$. Therefore, $g$ has nonzero nontangential boundary values 
a.e. on $\mathbb T$, in particular, in some point $\zeta\in \mathbb T$. It is possible to construct an infinite Blaschke product $ \vartheta$ with zeros tend to $\zeta$ such that 
$$\inf_{z\in\mathbb D}(|g(z)|+| \vartheta(z)|)>0.$$ By the corona theorem, 
there exist $\varphi$, $\psi\in H^\infty$ such that $1=g\varphi+ \vartheta\psi$. 
Set $\mathcal M=\operatorname{clos} \vartheta(T)\mathcal H$. Clearly,
 $\operatorname{clos}X\mathcal M = \vartheta  H^2$ and  
$\operatorname{clos}P_{ \mathcal K_\vartheta}X\mathcal M ^\perp=  \mathcal K_\vartheta$. 
Therefore, $\dim \mathcal M^\perp=\infty$. If $T|_{\mathcal M}\approx S$, then $T\approx S$ by Lemma \ref{leminvers}.
\end{proof}

\begin{remark}Suppose that $T$ is an operator (on a Hilbert space), $T$ has no eigenvalues, $\mathcal M\in\operatorname{Lat}T$, $T|_{\mathcal M}\approx S$, 
 $\dim \mathcal M^\perp<\infty$, and $\sigma(P_{\mathcal M^\perp}T|_{\mathcal M^\perp})\subset\mathbb D$. Then 
 $T\approx S$. We left the proof to the reader. \end{remark}

\subsection{Operators quasisimilar to $U_{\mathbb T}$}

The result of this subsection  appears as an answer on Vasily Vasyunin's  twenty-years old question (private communication) 
originally concerning contractions with scalar outer characteristic functions. (On  characteristic functions of contractions see \cite{sfbk}.)

\smallskip

The following simple lemma is given for convenience of references.

  \begin{lemma}\label{lemcorona} Let $g\in H^\infty$. If $\inf_{\mathbb D}|g|=0$, 
then there exists a Blaschke product $\vartheta$ such that $\inf_{\mathbb D}(|g|+|\vartheta|)=0$.
\end{lemma}

Set $H^2_-=L^2\ominus H^2$ and  $S_\ast=P_{H^2_-}U_{\mathbb T}|_{H^2_-}$. 

The following lemma is actually well known.

 \begin{lemma}\label{leminvert} Suppose that $\vartheta$, $g\in H^\infty$, $\vartheta$ is inner,  and 
$\vartheta$  is coprime with the inner part of $g$. Then there exists $c>0$ such that 
\begin{equation} \label{cminus}
\|P_{H^2_-}g h\|\geq c\|h\| \ \ \text{ for every }\ h\in\overline\chi\overline{\mathcal K_\vartheta}\end{equation} 
if and only if 
\begin{equation} \label{corona}\inf_{\mathbb D}(|g|+|\vartheta|)>0.\end{equation} 
\end{lemma}

\begin{proof}
Clearly, $\|P_{H^2_-}gh\|=\|P_{H^2}\overline g\overline \chi \overline h\|$ for every $h\in L^2(\mathbb T)$. Therefore, 
\eqref{cminus} is equivalent to the left invertibility of $g(T_\vartheta)^\ast$. 
Since $\vartheta$ is coprime with the inner part of $g$, $\ker g(T_\vartheta)=\{0\}$. 
Thus, \eqref{cminus} is equivalent to the invertibility of $g(T_\vartheta)$. 
By the corona theorem,  \eqref{corona} is equivalent to the existence of functions $\varphi$, $\psi\in H^\infty$ such that 
\begin{equation} \label{corona1}g\varphi+\vartheta\psi=1.\end{equation} 

If $g(T_\vartheta)$ is invertible, then there exists 
$\varphi\in H^\infty$ such that $\varphi(T_\vartheta)=g(T_\vartheta)^{-1}$ 
{\cite[Theorem X.2.10]{sfbk}}. Consequently, there exists $\psi\in H^\infty$ 
such that \eqref{corona1} is fulfilled, and \eqref{corona} follows. Conversely, if  \eqref{corona1} is fulfilled for some 
$\varphi$, $\psi\in H^\infty$, then  $\varphi(T_\vartheta)=g(T_\vartheta)^{-1}$, and \eqref{cminus} follows. 
\end{proof}

The following lemma is a corollary of {\cite[Theorem 3]{ker89}}.

 \begin{lemma}\label{lemsimshift} Suppose that $T$ is a power bounded operator (on a Hilbert space), $T\sim U_{\mathbb T}$, 
$\mathcal N\in\operatorname{Lat}T$, and $T|_{\mathcal N}\approx S$. 
Then $P_{\mathcal N^\perp}T|_{\mathcal N^\perp}$ is of class $C_{01}$. 
\end{lemma}

\begin{proof} Since $T$ is of class $C_{11}$, $P_{\mathcal M^\perp}T|_{\mathcal M^\perp}$ is of class $C_{\cdot 1}$ 
for every $\mathcal M\in\operatorname{Lat}T$. Since $T|_{\mathcal N}\approx S$, 
 $(T|_{\mathcal N})^{(a)}\cong U_{\mathbb T}$, see  \cite{ker89}.
 Since $T^{(a)}\cong U_{\mathbb T}$, 
by {\cite[Theorem 3]{ker89}},  
$(P_{\mathcal N^\perp}T|_{\mathcal N^\perp})^{(a)} = \mathbb O$. 
It means that $P_{\mathcal N^\perp}T|_{\mathcal N^\perp}$ is of class $C_{0\cdot}$.
\end{proof}

\begin{proposition}\label{proprgsim} 
For $g\in H^\infty$ set 
\begin{equation}\label{rg} R_g=\begin{pmatrix} S & (\cdot,\overline\chi\overline g)\text{\rm \bf{1}} 
\\ \mathbb O & S_\ast \end{pmatrix}.\end{equation}
Then $R_g$ is power bounded, and $R_g\sim U_{\mathbb T}$ if and only if $g$ is outer. Moreover, the following are equivalent:
\begin{enumerate}[\upshape (i)]
\item  $R_g\approx U_{\mathbb T}$;
\item $1/g\in H^\infty$;
\item $R_g\sim U_{\mathbb T}$  and if $\mathcal M\in\operatorname{Lat}R_g$ is such that $R_g|_{\mathcal M}$ is of class $C_{10}$, then 
 $R_g|_{\mathcal M}\approx S$.\end{enumerate}\end{proposition}

\begin{proof} Easy calculation shows that 
$$R_g^n(h_1+h_2)=\chi^nh_1+P_{H^2}\chi^n P_{H^2_-}gh_2+P_{H^2_-}\chi^nh_2$$
 for every  $ n\in\mathbb N$, $ h_1\in H^2$  and $ h_2\in H^2_-$. 
 Therefore, 
\begin{equation}\label{rglim} \lim_n\|R_g^n(h_1+h_2)\|^2=\|h_1\|^2+\|P_{H^2_-}gh_2\|^2 
\ \  \text{ for every } \ h_1\in H^2 \text{  and }  h_2\in H^2_-.\end{equation}
If $g$ has a nontrivial inner part, say  $\vartheta$, then $P_{H^2_-}gh=0$ 
for every $h\in\overline\chi\overline{\mathcal K_\vartheta}$. It follows from \eqref{rglim} 
that such $R_g$ can not be of class $C_{1\cdot}$. Therefore, if $R_g\sim U_{\mathbb T}$, then $g$ is outer. 
Conversely, if $g$ is outer, then the quasiaffinities 
\begin{equation}\label{rgsim}\begin{pmatrix} I_{H^2} & \mathbb O \\ 
\mathbb O &  g(S_\ast)\end{pmatrix} 
\ \ \ \text{ and } \ \ \ 
\begin{pmatrix} g(S) & P_{H^2}g(U_{\mathbb T})|_{H^2_-} \\ \mathbb O & I_{H^2_-} \end{pmatrix}\end{equation}
realize the relation $R_g\sim U_{\mathbb T}$. Moreover, their product is $g(U_{\mathbb T})$. 
If $1/g\in H^\infty$, the quasiaffinities from \eqref{rgsim} are invertible. Thus, (ii)$\Rightarrow$(i). 

The relation (i)$\Rightarrow$(iii) is trivial. Suppose that (iii) is fulfilled. Then  $g$ is outer 
(as was proved just above).  Let $\vartheta\in H^\infty$ be an inner function. Then 
$H^2\oplus\overline\chi\overline{\mathcal K_\vartheta}\in\operatorname{Lat}R_g$, and 
$R_g|_{H^2\oplus\overline\chi\overline{\mathcal K_\vartheta}}$ is of class $C_{10}$. Indeed,  
$R_g|_{H^2\oplus\overline\chi\overline{\mathcal K_\vartheta}}$ is of class $C_{1\cdot}$, 
because $R_g\sim U_{\mathbb T}$. Furthermore, 
$$R_g|_{H^2\oplus\overline\chi\overline{\mathcal K_\vartheta}}=\begin{pmatrix} S & \ast 
\\ \mathbb O & P_{H^2_-}S_\ast|_{\overline\chi\overline{\mathcal K_\vartheta}} \end{pmatrix},$$
and $P_{H^2_-}S_\ast|_{\overline\chi\overline{\mathcal K_\vartheta}}\cong T_\vartheta$. 
In particular,  $P_{H^2_-}S_\ast|_{\overline\chi\overline{\mathcal K_\vartheta}}$ 
is of class $C_{00}$ {\cite[Proposition III.4.2]{sfbk}}. Clearly, $S$ is of class $C_{10}$. Consequently, 
$R_g|_{H^2\oplus\overline\chi\overline{\mathcal K_\vartheta}}$ is of class $C_{\cdot 0}$ by {\cite[Theorem 3]{ker89}} 
(applied to adjoint). 

Now suppose that $R_g|_{H^2\oplus\overline\chi\overline{\mathcal K_\vartheta}}\approx S$. Then there exists $c>0$ such that 
$\|R_g^n x\|\geq c\|x\|$ for every $n\in\mathbb N$ and every $x\in H^2\oplus\overline\chi\overline{\mathcal K_\vartheta}$. 
It follows from \eqref{rglim} that \eqref{cminus} is fulfilled for $g$ and  $\vartheta$. By Lemma \ref{leminvert}, 
\eqref{corona} is fulfilled for $g$ and $\vartheta$.

Since  $\vartheta$ was an  arbitrary inner function, $1/g\in H^\infty$  by Lemma \ref{lemcorona}.  Thus, (iii)$\Rightarrow$(ii). 
\end{proof}

\begin{remark}\label{remcassier} For $g\in H^\infty$ let $ R_g$ be defined by \eqref{rg}. Then $R_g$ is similar to a contraction by {\cite[Corollary 4.2]{cassier}}.
\end{remark}

\begin{proposition}\label{proprg} Suppose that $T$ is a power bounded operator (on a Hilbert space), $T\sim U_{\mathbb T}$, and there exists 
$\mathcal N\in\operatorname{Lat}T$ such that $T|_{\mathcal N}\approx S$ and $T^\ast|_{\mathcal N^\perp}\approx S$. Then there exists 
$g\in H^\infty$ such that $T\approx R_g$, where $R_g$ is defined by \eqref{rg}.\end{proposition}

\begin{proof} Denote by $\mathcal H$ the space on which $T$ acts and 
by $X_+$ the canonical intertwining mapping for $T$ and $U_{\mathbb T}$.    
Since $T|_{\mathcal N}\approx S$, we have that $X_+$ is left invertible on $\mathcal N$. 
Without loss of generality we can assume that $X_+\mathcal N=H^2$. Set  
$$Y=P_{H^2_-}X_+|_{\mathcal N^\perp} \ \ \ \ \text{and } \ \ \ 
Z=\begin{pmatrix} X_+|_{\mathcal N} & P_{H^2}X_+|_{\mathcal N^\perp} \\ \mathbb O & I_{\mathcal N^\perp} \end{pmatrix}.$$
Then $Y\in\mathcal L(\mathcal N^\perp,H^2_-)$, $YP_{\mathcal N^\perp}T|_{\mathcal N^\perp}=S_\ast Y$,
$Z\in \mathcal L(\mathcal H, H^2\oplus\mathcal N^\perp)$, $Z$ is invertible, and 
$$ ZTZ^{-1}=\begin{pmatrix} S & P_{H^2}U_{\mathbb T}|_{H^2_-}Y 
\\ \mathbb O & P_{\mathcal N^\perp}T|_{\mathcal N^\perp}\end{pmatrix}.$$
Since $S\cong (S_\ast)^\ast$ and $T^\ast|_{\mathcal N^\perp}\approx S$, 
there exists $Z_1\in \mathcal L(\mathcal N^\perp,H^2_-)$ such that $Z_1$ is invertible 
and $Z_1P_{\mathcal N^\perp}T|_{\mathcal N^\perp}Z_1^{-1}=S_\ast$. We have 
$$\begin{pmatrix}I_{H^2} & \mathbb O \\ \mathbb O & Z_1\end{pmatrix}
\begin{pmatrix} S & P_{H^2}U_{\mathbb T}|_{H^2_-}Y 
\\ \mathbb O & P_{\mathcal N^\perp}T|_{\mathcal N^\perp}\end{pmatrix}
\begin{pmatrix}I_{H^2}&\mathbb O \\ \mathbb O & Z_1^{-1}\end{pmatrix}=
\begin{pmatrix} S & P_{H^2}U_{\mathbb T}|_{H^2_-}Y Z_1^{-1}
\\ \mathbb O & S_\ast\end{pmatrix}.$$
Since $YZ_1^{-1}$ commutes with $S_\ast$, there exists $g\in H^\infty$ 
such that $YZ_1^{-1}h=P_{H^2_-}gh$ for every $h\in H^2_-$. Therefore, 
$$P_{H^2}U_{\mathbb T}|_{H^2_-} Y Z_1^{-1} = (\cdot,\overline\chi\overline g)\text{\rm \bf{1}}.\qedhere$$
\end{proof}

\begin{theorem}\label{thtrg}
 Suppose that $T$ is a power bounded operator (on a Hilbert space), $T\sim U_{\mathbb T}$, and there exists 
$\mathcal N\in\operatorname{Lat}T$ such that $T^\ast|_{\mathcal N^\perp}\approx S$. 
Furthermore, suppose that $T|_{\mathcal M}\approx S$ for every $\mathcal M\in\operatorname{Lat}T$ 
such that $T|_{\mathcal M}$ is of class $C_{10}$. Then $T\approx U_{\mathbb T}$.\end{theorem}

\begin{proof} By Lemma \ref{lemsimshift} applied to $T^\ast$, $T|_{\mathcal N}$ is of class $C_{10}$. By assumption, 
$T|_{\mathcal N}\approx S$. By Proposition \ref{proprg}, $T\approx R_g$ for some $g\in H^\infty$. By Proposition \ref{proprgsim}, 
$R_g\approx U_{\mathbb T}$.
\end{proof}

\begin{corollary} Suppose that $T$ is a  polynomially bounded operator (on a Hilbert space), $T\sim U_{\mathbb T}$, and  
 $T|_{\mathcal M}\approx S$ for every $\mathcal M\in\operatorname{Lat}T$ 
such that $T|_{\mathcal M}$ is of class $C_{10}$. Then $T\approx U_{\mathbb T}$.
\end{corollary}

\begin{proof} The operator $T$ satisfies to assumptions of Theorem \ref{thtrg}  by {\cite[Theorem 1.1]{gam19}}  
 (applied to $T^\ast$).
\end{proof}

\end{document}